\documentclass[12pt,leqno,twoside,sort]{amsart}

\usepackage{amssymb,amsmath,amsthm,soul,color}

\usepackage[normalem]{ulem}

\usepackage{amssymb,amsmath,amsthm}

\usepackage{mathrsfs}

\usepackage[utf8]{inputenc}

\usepackage{todonotes}
\setuptodonotes{inline}

\usepackage[left=2cm, right=2cm, top=2cm, bottom=2cm]{geometry}

\usepackage{url}

\usepackage[numbers,sort&compress]{natbib}

\usepackage{fixmath}

\usepackage[T1]{fontenc}
\usepackage{lmodern}
\usepackage{microtype}
\usepackage[normalem]{ulem}
\usepackage{mathtools}

\usepackage[colorlinks=true,urlcolor=blue,
citecolor=red,linkcolor=blue,linktocpage,pdfpagelabels,
bookmarksnumbered,bookmarksopen]{hyperref}

\usepackage{bbm}

\usepackage{xcolor,cancel}

\newtheorem{Th}{Theorem}[section]

\newtheorem{Def}[Th]{Definition}

\newcommand{\R}{\mathbb{R}}

\newcommand{\cC}{{\mathcal C}}

\newcommand{\cH}{{\mathcal H}}

\newcommand{\weakto}{\rightharpoonup}

\numberwithin{equation}{section}

\newcommand{\supp}{\mathrm{supp}\,}

\newcommand{\loc}{\mathrm{loc}}
\newcommand{\rad}{\mathrm{rad}}
\newcommand{\JX}{J_{\mid X}}

\newcommand{\vertiii}[1]{{\left\vert\kern-0.25ex\left\vert\kern-0.25ex\left\vert #1 
    \right\vert\kern-0.25ex\right\vert\kern-0.25ex\right\vert}}
    
\title[Multiplicity of solutions for nonlinear scalar field equations]{Note on the multiplicity of solutions for nonlinear scalar field equations with a critical inverse-square potential}
\author{Bartosz Bieganowski, Daniel Strzelecki}

\address[B. Bieganowski]{\newline\indent
	Faculty of Mathematics, Informatics and Mechanics, \newline\indent
	University of Warsaw, \newline\indent
	ul. Banacha 2, 02-097 Warsaw, Poland}
\email{\href{mailto:bartoszb@mimuw.edu.pl}{bartoszb@mimuw.edu.pl}}

\address[D. Strzelecki]{\newline\indent
	Faculty of Mathematics, Informatics and Mechanics, \newline\indent
	University of Warsaw, \newline\indent
	ul. Banacha 2, 02-097 Warsaw, Poland}
\email{\href{mailto:dstrzelecki@mimuw.edu.pl}{dstrzelecki@mimuw.edu.pl}}

\date{}

\begin{document}

\begin{abstract}
We are interested in the multiplicity of solutions to the following scalar field equation
$$
-\Delta u - \frac{(N-2)^2}{4|x|^2} u  = g(u), \quad \mbox{in } \mathbb{R}^N \setminus \{0\}.
$$
We establish the existence of infinitely many radial and non-radial solutions.

\medskip

\noindent \textbf{Keywords:} variational methods, singular potential, scalar field equation, critical Hardy potential
   
\noindent \textbf{AMS Subject Classification:} 35Q99, 35J10, 35J20, 58E99
\end{abstract}

\maketitle

\section{Introduction}

We study the following scalar field equation
\begin{equation}\label{eq:main}
-\Delta u - \frac{(N-2)^2}{4|x|^2} u  = g(u),  \quad \mbox{in } \R^N \setminus \{0\},
\end{equation}
where $N \geq 3$, and we are interested in the multiplicity of solutions. 
Equation \eqref{eq:main} involves the inverse-square Hardy potential with the optimal constant, which places the problem at the critical threshold of the Hardy inequality and leads to substantial analytical difficulties.

The nonlinear scalar field equation
$$
-\Delta u = g(u)
$$
is well understood, starting from the seminal works of Berestycki and Lions \cite{BL, BL2}, who established existence results under general assumptions on $g$, including the celebrated Berestycki-Lions conditions. 
Further developments addressed qualitative properties of solutions as well as symmetry and uniqueness issues.
Later, the existence and multiplicity of nonradial solutions were obtained in \cite{JeanjeanLu} and \cite{MR4173560}, among others. 

The literature is far more sparse regarding the critical inverse-square Hardy potential
$- \frac{(N-2)^2}{4|x|^2}$ is present with the optimal constant, usually referred to as the \emph{critical Hardy potential}. 
To the best of our knowledge, the first contribution addressing the existence of solutions for the pure power nonlinearity $g(u) = |u|^{p-2}u$, $2 < p < 2^* := \frac{2N}{N-2}$, is due to Mukherjee et al.\ \cite{Mukherjee}, who obtained existence results by exploiting a Hardy-Gagliardo-Nirenberg inequality and minimizing the associated Weinstein functional. The inverse-square potential $V(x) = -\frac{\mu}{|x|^2}$ falls into the class of the so-called \textit{singular transitional potentials} \cite{RevModPhys.43.36} and it is studied by many authors from the mathematical and physical point of view \cite{PhysRev.80.797, refId0, MR3098947, GuoMederski, PhysRevA.76.032112, narnhofer_2024_pxv87-rfh51, Deng, LiLiTang}.

One of the main difficulties in studying \eqref{eq:main} is the lack of coercivity of the quadratic form on $H^1(\R^N)$ generated by the operator
$-\Delta - \frac{(N-2)^2}{4|x|^2}$. 
As a consequence, the natural energy space associated with \eqref{eq:main} is strictly larger than $H^1(\R^N)$ and cannot be characterized as a standard Sobolev space.
Such a space was introduced independently in \cite{Mukherjee, Suzuki, TrachanasZographopoulos}, where some of its fundamental properties were established.

A major additional difficulty is that this energy space is not translation invariant, which renders the classical concentration-compactness principle of Lions inapplicable in a direct way. 
In a recent contribution \cite{BiegStrz}, the functional framework was refined and a suitable version of the concentration-compactness approach was developed, leading to the existence of a nontrivial weak solution to \eqref{eq:main}. 
It is also shown therein that working in the enlarged energy space is essential: if a solution is sufficiently regular (for instance, nonnegative and of class $\cC^2$), then it does not belong to $H^1(\R^N)$, cf. \cite{Smets}.

The solution constructed in \cite{BiegStrz} is obtained as a minimizer of the energy functional on the Poho\v{z}aev manifold and is radial provided that $g$ is odd. 
However, due to the lack of full regularity theory in the critical Hardy setting, it is not known whether all weak solutions satisfy the Poho\v{z}aev identity.
It is also shown in \cite{BiegStrz} that equation \eqref{eq:main} admits at least one nonradial solution.

In the present paper, we establish the multiplicity of both radial and nonradial solutions to \eqref{eq:main}. 
Our approach builds upon the functional framework developed in \cite{BiegStrz} and exploits the additional symmetry induced by the oddness of the nonlinearity, by the application of the abstract critical point theorem due to Ikoma \cite{Ikoma}.
Accordingly, we work under the same assumptions as in \cite{BiegStrz}, with the additional requirement that the right-hand side is odd.

We impose the following general conditions on the nonlinearity $g$, already introduced in \cite{BiegStrz}:
\begin{itemize}
\item[(G1)] $g : \R \rightarrow \R$ is continuous and odd;
\item[(G2)] $-\infty < \liminf_{s \to 0}  \frac{g(s)}{s} \leq \limsup_{s \to 0} \frac{g(s)}{s} = -\omega < 0$;
\item[(G3)] $\lim_{|s| \to \infty} \frac{g(s)}{|s|^{p-1}} = 0$ and $\lim_{|s| \to \infty} g(s)s \geq 0$ for some $p \in \left( 2,2^* \right)$;
\item[(G4)] there exists $\zeta_0 > 0$ such that $G(\zeta_0) > 0$, where $G(s) := \int_0^{s} g(\tau) \, d\tau$;
\item[(G5)] there exist $\nu > 2$ and $\gamma_0 < 0$ such that
\[
G(s) - \frac{1}{\nu} g(s)s \leq \gamma_0 s^2 \quad \text{for all } s\in\R.
\]
\end{itemize}

Under these assumptions, the solutions of \eqref{eq:main} can be interpreted as the so-called \textit{standing waves} $\Psi(t,x) = e^{i \omega t} u(x)$ of the nonlinear Schr\"odinger equation with a critical inverse-square potential
$$
i \partial_t \Psi = -\Delta \Psi - \frac{(N-2)^2}{4 |x|^2} \Psi - f(|\Psi|)\Psi, \quad (t,x) \in \R \times \R^N,
$$
where $g(u) = \omega u + f(|u|)u$. The critical constant $\mu_c := \frac{(N-2)^2}{4}$ represents the threshold for the coerciveness of the quadratic form associated with the Hamiltonian $\cH_\mu := -\Delta - \frac{\mu}{|x|^2}$. More precisely, if $\mu < \mu_c$, the corresponding quadratic form is coercive on the Sobolev space $H^1(\R^N)$ and power functions $\psi_{\pm} = |x|^{\frac{-(N-2) \pm \sqrt{(N-2)^2 -4 \mu}}{2}}$ belonging to $L^2_{\loc} (\R^N)$ satisfy $\cH_\mu \psi_{\pm} = 0$. If $\mu = \mu_c$, we have $\cH_{\mu_c} \left( |x|^{-\frac{N-2}{2}} \right) = 0$ and $\cH_{\mu_c} \left( |x|^{-\frac{N-2}{2}} \ln |x| \right) = 0$. In contrast, when $\mu > \mu_c$, coerciveness is lost, $\sqrt{(N-2)^2 -4 \mu}$ becomes imaginary and the system exhibits the so-called \textit{fall to the center} phenomenon, describing a quantum-mechanical instability driven by the strong singular attraction at the origin (see e.g. \cite{Plestid2018,Derezinski}).

The paper is organized as follows. In Section \ref{sect:2}, we introduce the functional and variational framework, including the definition of the energy space and a discussion of its main properties, following \cite{BiegStrz} and the abstract critical point theory developed in \cite{Ikoma}. In Section \ref{sect:3}, we state the main result and provide its proof.

\section{Functional and variational setting}\label{sect:2}

\subsection{The \texorpdfstring{$X^1$}{X1} space}

Recall the classical Hardy inequality
\begin{equation} \label{ineq:Hardy}
    \int_{\R^N} \frac{u^2}{|x|^2} \, dx \leq \frac{4}{(N-2)^2} \int_{\R^N} |\nabla u|^2 \, dx, \quad u \in H^1 (\R^N).
\end{equation}
This inequality implies that the functional
\begin{equation}\label{normX1-oryg}
\| u \|^2 := \int_{\R^N} |\nabla u|^2 \, dx - \frac{(N-2)^2}{4} \int_{\R^N} \frac{u^2}{|x|^2} \, dx + \int_{\R^N} u^2 \, dx, \quad u \in H^1 (\R^N),
\end{equation}
defines a norm on $H^1 (\R^N)$. However, due to the optimality of the constant $\frac{(N-2)^2}{4}$ and the fact that minimizers of \eqref{ineq:Hardy} do not belong to $H^1(\R^N)$, this norm is not complete in $H^1 (\R^N)$. To work within a Hilbert space framework, we consider $X^1 (\R^N)$ as the completion of $H^1 (\R^N)$ with respect to \eqref{normX1-oryg}. This space has been investigated in \cite{Suzuki, Mukherjee, TrachanasZographopoulos}.
According to \cite[Theorem 1.2]{Frank} (see also \cite{Suzuki}), we have the continuous embeddings
\begin{equation}\label{embedHs}
H^1 (\R^N) \subset X^1 (\R^N) \subset H^s(\R^N)
\end{equation}
for any $s \in [0, 1)$, where $H^0 (\R^N) := L^2 (\R^N)$. Consequently,
\begin{equation}\label{eq:embeddings}
    X^1 (\R^N) \subset L^t (\R^N)
\end{equation}
is continuous for $t \in [2, 2^*)$ and locally compact (specifically, the embedding $X^1 (\R^N) \subset L^t_{\loc} (\R^N)$ is compact).

For functions $u \in X^1 (\R^N) \setminus H^1 (\R^N)$, the first two integrals in \eqref{normX1-oryg} are not well-defined. As discussed in \cite{TrachanasZographopoulos, VazquezZ}, the inner product structure on $X^1 (\R^N)$ requires a generalized characterization. We introduce the singular bilinear form $\xi(\cdot, \cdot)$ defined by
\begin{equation}\label{xi:uv}
    \xi(u, v) := \lim_{\varepsilon \to 0+} \left( \int_{|x| > \varepsilon} \nabla u \nabla v \, dx - \frac{(N-2)^2}{4} \int_{|x| > \varepsilon} \frac{u v}{|x|^2} \, dx - \frac{N-2}{2} \varepsilon^{-1} \int_{|x|=\varepsilon} u v \, dS \right).
\end{equation}
We also denote $\xi(u) := \xi(u, u)$. The scalar product $\langle \cdot, \cdot \rangle$ and the norm on $X^1 (\R^N)$ are then given by
\begin{equation}\label{normX1}
    \langle u,v \rangle = \xi(u, v) + \int_{\R^N} uv \, dx, \quad
    \|u\|^2 = \xi(u) + \int_{\R^N} u^2 \, dx.
\end{equation}
Note that for $u \in H^1 (\R^N)$, \eqref{normX1} reduces to the standard formula \eqref{normX1-oryg}.

\subsection{Symmetry classes.}

To study nonradial solutions, we assume $N\geq 4$ and $N \neq 5$. Let $2 \leq M \leq \frac{N}{2}$ such that $N-2M\neq 1$. We decompose $\R^N = \R^M \times \R^M \times \R^{N-2M}$ and write points as $x=(x_1, x_2, x_3)$ with $x_1, x_2 \in \R^M$ and $x_3 \in \R^{N-2M}$. 
We introduce the involution $\tau : \R^N \rightarrow \R^N$ given by
$$
\tau(x_1, x_2, x_3) := (x_2, x_1, x_3),
$$
and the corresponding subspace of antisymmetric functions
\begin{equation}\label{eq:Xtau}
    X_\tau := \left\{ u \in X^1 (\R^N) : u(x) = - u(\tau x) \text{ for a.e. } x \in \R^N \right\}.
\end{equation}
Observe that $X_\tau$ contains no nontrivial radially symmetric functions.
Additionally, let $\mathcal{O}_2 := \mathcal{O}(M) \times \mathcal{O}(M) \times \mathcal{O}(N-2M) \subset \mathcal{O}(N)$ and let $X^1_{\mathcal{O}_2} (\R^N)\subset X^1(\R^N)$ denote the subspace of $\mathcal{O}_2$-invariant functions.

To look for radially symmetric solutions, we define $X^1_\rad(\R^N)\subset X^1(\R^N)$ as the subspace of radially symmetric functions, as long as $N \geq 3$.

In what follows, we define the working space $X$ as
$$
X := \left\{ \begin{array}{ll}
X^1_{\rad} (\R^N)      &  \text{(radial case: $N\geq 3$),}\\
X^1_{\mathcal{O}_2} (\R^N) \cap X_\tau      &  \text{(nonradial case: $N \geq 4$, $N \neq 5$).}
\end{array} \right.
$$
Combining the continuous embeddings \eqref{embedHs} with compactness results from \cite[Theorems II.1 and III.3]{Lions4}, we conclude that the embedding
$$
X \subset L^p (\R^N), \quad p \in (2,2^*)
$$
is compact.

\subsection{Variational setting}
We define the energy functional $J : X^1 (\R^N) \rightarrow \R$ by
$$
J(u) = \frac12 \xi(u) - \int_{\R^N} G(u) \, dx.
$$
Under assumptions (G1)--(G3) and in view of the embeddings \eqref{eq:embeddings}, the functional $J$ is of class $\cC^1$ and its derivative is given by
$$
J'(u)(v) = \xi(u,v) - \int_{\R^N} g(u) v \, dx, \quad u,v \in X^1(\R^N).
$$
It is established that any critical point $u \in X^1 (\R^N)$ of $J$ is a weak solution to \eqref{eq:main} on $\R^N \setminus \{0\}$ (see \cite{BiegStrz}).

We introduce the Poho\v{z}aev functional $P : X^1 (\R^N) \rightarrow \R$ defined as
\[
    P(u) = \xi(u) - 2^* \int_{\R^N} G(u) \, dx.
\]
Due to the presence of singular terms in $\xi(u)$, it remains an open question whether every critical point $u \in X^1(\R^N)$ of $J$ satisfies the identity $P(u)=0$. Nevertheless, we will search specifically for solutions that satisfy this condition.

\begin{Def}\label{DEf:PPS}
Let $c \in \R$. We say that $J:X\to \R$ satisfies $(\mathrm{PPS})_c$ if every sequence $(u_n)\subset X$ such that 
$$
J(u_n) \to c, \quad J'(u_n) \to 0, \quad P(u_n) \to 0,
$$
has a convergent subsequence.
\end{Def}

Let us define the scaling action $\Phi : \R \times X^1 (\R^N) \rightarrow X^1 (\R^N)$ by
$$
\Phi(t, u)(x) := u(e^{-t} x).
$$
A direct computation shows that
$$
J(\Phi(t, u)) = \frac{e^{(N-2)t}}{2} \xi(u) - e^{Nt} \int_{\R^N} G(u) \, dx,
$$
and the map $J \circ \Phi$ is of class $\cC^1 (\R \times X^{1}(\R^N))$.

In this framework, the abstract critical point theorem by Ikoma \cite[Theorem 2.4]{Ikoma} takes the following form.

\begin{Th}\label{T:Ikoma}
  Suppose that $\JX \in \cC^1(X, \R)$ is an even functional, i.e., $\JX(-u) = \JX(u)$ for all $u \in X$. Assume further that:
\begin{enumerate}
    \item[(i)] $\JX$ satisfies the $(\mathrm{PPS})_c$ condition for each $c > 0$;
    \item[(ii)] there exists $\rho_0 > 0$ such that $\inf_{\|u\|=\rho_0} \JX(u) > 0 = J(0)$;
    \item[(iii)] for each $k \geq 1$, there exists an odd map $\gamma_k \in \cC( S^{k-1}, X)$ such that 
    $$
    \max_{\sigma \in S^{k-1}} J(\gamma_k(\sigma)) < 0 
    \quad \text{and} \quad 
    \gamma_k(S^{k-1}) \subset X \setminus \overline{B_{\rho_0}(0)},
    $$
    where $S^{k-1} := \{ \sigma \in \R^k : |\sigma| = 1 \}$.
\end{enumerate}
Then there exists a sequence $(u_j)_{j=1}^\infty \subset X$ such that $(\JX)'(u_j) = 0$, $P(u_j)=0$, and $J(u_j) = c_j \to \infty$ as $j \to \infty$.
\end{Th}

Although the theorem provides critical points of the restricted functional $\JX$, by the principle of symmetric criticality (see \cite{Palais}), these points are also critical points of the unconstrained functional $J$.

\section{Main result}\label{sect:3}

\begin{Th}
Suppose that (G1)--(G5) hold. Then, 
\begin{itemize}
    \item [(i)] if $N\geq 3 $, there exists a sequence of radial weak solutions $(u_j)_{j=1}^\infty \subset X$ of \eqref{eq:main},
    \item [(ii)] if $N=4$ or $N\geq 6$, there exists a sequence of non-radial weak solutions $(u_j)_{j=1}^\infty \subset X$ of \eqref{eq:main}.
\end{itemize}
Moreover, in both cases, $u_j \in C^{1,\alpha}_{\loc} (\R^N \setminus \{0\})$ for some $\alpha \in (0,1)$, $P(u_j) = 0$ and $J(u_j) \to \infty$ as $j \to \infty$. If, in addition, $g$ is H\"older continuous, $u_j \in C^2 (\R^N \setminus \{0\})$.
\end{Th}

\begin{proof}
It is enough to verify assumptions (i)--(iii) of Theorem \ref{T:Ikoma}, while the regularity of solutions follows directly from \cite[Theorem 1.2 (c),(d)]{BiegStrz}.

To show (i) take $(u_n)\subset X$ such that $J(u_n) \to c > 0, \, J'(u_n) \to 0$, and $P(u_n) \to 0$. From \cite[Lemma 5.1]{BiegStrz}, $(u_n)$ is bounded in $X^1 (\R^N)$. Then $u_n\weakto u_0$ in $X^1 (\R^N)$, up to a subsequence. By the compactness of embeddings $u_n\to u_0$ in $L^p(\R^N)$, $p\in(2,2^*)$. Since $J$ is weak-to-weak* continuous, $J'(u_0) = 0$. Then
\begin{align*}
o(1) &= J'(u_n)(u_n-u_0) - J'(u_0)(u_n-u_0) = \xi(u_n - u_0) - \int_{\R^N} \left( g(u_n) - g(u_0) \right) (u_n - u_0) \, dx \\
&= \xi(u_n - u_0) +\omega \int_{\R^N} (u_n-u_0)^2 \, dx \\
&\quad - \int_{\R^N} \left( g(u_n) + \omega u_n \right) (u_n-u_0) \, dx + \int_{\R^N} \left( g(u_0) + \omega u_0 \right) (u_n - u_0) \, dx.
\end{align*}
Observe that, by (G1)--(G3), for every $\varepsilon > 0$ there is $C_\varepsilon>0$ such that
\begin{align*}
\left| \int_{\R^N} \left( g(u_n) + \omega u_n \right) (u_n-u_0) \, dx \right| &\leq \varepsilon \int_{\R^N} |u_n| |u_n-u_0| \, dx + C_\varepsilon \int_{\R^N} |u_n|^{p-1} |u_n-u_0| \, dx \\
&\lesssim \varepsilon + C_\varepsilon \|u_n - u_0\|_{L^p (\R^N)},
\end{align*}
where we used that $(u_n)$ is bounded in $L^2 (\R^N)$ and in $L^p(\R^N)$. Since $\|u_n - u_0\|_{L^p (\R^N)} \to 0$ and $\varepsilon > 0$ was arbitrary, we obtain that
$$
\int_{\R^N} \left( g(u_n) + \omega u_n \right) (u_n-u_0) \, dx \to 0.
$$
Similarly
$$
\int_{\R^N} \left( g(u_0) + \omega u_0 \right) (u_n-u_0) \, dx \to 0.
$$
Hence $\xi(u_n - u_0) + \omega \int_{\R^N} (u_n-u_0)^2 \, dx \to 0$, namely $u_n \to u_0$ in $X^1 (\R^N)$. Hence, $(\mathrm{PPS})_c$ holds.

Condition (ii) follows directly from \cite[Lemma 3.1]{BiegStrz}, and $\rho_0$ can be chosen arbitrarily small.

To show (iii), we claim that for every $k \geq 1$ there are a radius $R_k > 0$ and an odd map $\overline{\gamma}_k \in \cC (S^{k-1}, X)$ such that $\supp (\overline{\gamma}_k (\sigma)) \subset B(0,R_k)$, $\int_{\R^N} G \left( \overline{\gamma}_k (\sigma) \right) \, dx \geq 1$. In the radial setting it follows from \cite[Theorem 10]{BL2}, while for the nonradial setting it is a consequence of \cite[Lemma 4.2]{JeanjeanLu}. Then, we define $\gamma_k$ as $\gamma_k(\sigma)(x) = \overline{\gamma}_k (\sigma) (e^{-\theta_k} x)$, where $\theta_k > 0$ is to be chosen. Note that, then
\begin{align*}
J(\gamma_k(\sigma)) &= \frac{1}{2} \xi(\gamma_k(\sigma)) - \int_{\R^N} G(\gamma_k(\sigma)) \, dx = \frac12 e^{(N-2) \theta_k} \xi ( \overline{\gamma}_k (\sigma)) - e^{N\theta_k} \int_{\R^N} G( \overline{\gamma}_k (\sigma)) \, dx \\
&\leq \frac12 e^{(N-2) \theta_k} \xi ( \overline{\gamma}_k (\sigma)) - e^{N\theta_k}.
\end{align*}
Hence, for sufficiently large $\theta_k > 0$, we get $\max_{\sigma \in S^{k-1}} J(\gamma_k(\sigma)) < 0$. Choosing sufficiently small $\rho_0$ in (ii), we obtain also that $\gamma_k(S^{k-1}) \subset X \setminus \overline{B_{\rho_0}(0)}$ and (iii) is satisfied.
\end{proof}

\section*{Acknowledgements}
Bartosz Bieganowski and Daniel Strzelecki were partly supported by the National Science Centre, Poland (Grant No. 2022/47/D/ST1/00487).
\bibliographystyle{acm}
\bibliography{Bibliography}
\end{document}